\documentclass[11pt,letterpaper]{article}
\usepackage[margin=1in]{geometry}
\usepackage{amsmath,amsthm,amssymb,enumerate,tikz}
\usepackage[hidelinks]{hyperref}
\hypersetup{
  pdftitle={Removable matchings in 2-connected graphs},
  pdfauthor={Jinik Hwang and Ringi Kim}
}
\newtheorem{theorem}{Theorem}[section]
\newtheorem{lemma}[theorem]{Lemma}
\newtheorem{corollary}[theorem]{Corollary}

\newtheorem{observation}[theorem]{Observation}
\newtheorem{conjecture}[theorem]{Conjecture}

\title{Removable matchings in $2$-connected graphs}
\author{Jinik Hwang\thanks{\raggedright Department of Mathematics, Inha University,
Incheon 22212, Republic of Korea. Email: \texttt{jnk.hwang@gmail.com}.}
\and Ringi Kim\thanks{\raggedright Corresponding author. Department of Mathematics,
Inha University, Incheon 22212, Republic of Korea.
Email: \texttt{ringikim@inha.ac.kr}.}}
\date{}

\begin{document}
\maketitle

\begin{abstract}
A matching $M$ of a $2$-connected graph $G$ is \emph{removable}
if $G-M$ is $2$-connected.
We prove that, for every integer $d\ge4$, every $2$-connected
graph with at least $2d$ vertices and minimum degree at least $d$
has a removable matching of size $d$.
This improves the previously known lower bound $d-2$ for $d\ge5$
to the optimal value $d$, answering a question of Li, Zhou,
Fujita, and Mao.
We also show that a $2$-connected graph on at least six vertices
with minimum degree at least $3$ has a removable matching
of size $3$ if and only if it is not a wheel.
Our proofs use edge contraction and the structure of
\emph{minimal non-removable matchings}.
\end{abstract}

%=======================================================================
\section{Introduction}
%=======================================================================

All graphs are finite, undirected, and simple.
For a graph $G$, let $V(G)$ and $E(G)$ be its vertex and
edge sets, respectively, and let $\delta(G)$ be its minimum degree.
For $F\subseteq E(G)$, let $G-F$ be the graph obtained by
deleting the edges in $F$, and write $G-e$ when $F=\{e\}$.

A classical problem in connectivity theory is to determine which
subgraphs can be deleted while preserving connectivity. An edge $e$
of a $k$-connected graph $G$ is \emph{$k$-removable} if $G-e$ is
$k$-connected. Halin~\cite{halin} proved that every $k$-connected
graph with minimum degree at least $k+1$ contains a $k$-removable
edge. The degree bound is sharp, since no edge of a $k$-regular
$k$-connected graph is $k$-removable. 
More generally, connectivity-keeping subgraphs have been
studied extensively~\cite{ckl,fk,hl,llh,mader-path,mader-tree,yt}.
See the survey of Tian and Meng~\cite{tm}.

In the edge-deletion setting, Hasunuma~\cite{hasunuma} conjectured
that, for every tree $T$ of order $m$, every $k$-connected graph $G$
with $\delta(G)\ge k+m-1$ contains a copy $T'$ of $T$ such that
$G-E(T')$ is $k$-connected. This conjecture and its edge-connectivity
analogue were recently proved by Clay and Jord\'an~\cite{cj}.

A matching $M$ of a $k$-connected graph $G$ is \emph{$k$-removable}
if $G-M$ is $k$-connected. For an integer $\ell\ge1$, an
\emph{$\ell$-matching} is a matching of size $\ell$.
Matchings are
natural candidates for connectivity-keeping subgraphs because their
deletion decreases the degree of each vertex by at most one.
Li, Zhou, Fujita, and Mao~\cite{lzfm} introduced the following
function. For integers $d>k\ge1$, let $f(k,d)$ be the largest integer
$m$ such that every $k$-connected graph $G$ with $|V(G)|\ge2d$ and
$\delta(G)\ge d$ has a $k$-removable $m$-matching. Since
$K_{d,n-d}$ has no $(d+1)$-matching when $n\ge2d$, we have
$f(k,d)\le d$.
They asked for the exact value of $f(k,d)$.

Chu, Kim, and Park~\cite{ckp} proved that, 
for all positive integers $k$ and $m$,
every $k$-connected graph
$G$ on at least $2m$ vertices has a $k$-removable $m$-matching if
\[
\delta(G)\ge
\begin{cases}
\max\{k+\lceil m/2\rceil,\,2m\} & \text{if }k\ge m,\\[2pt]
k+m & \text{if }k<m.
\end{cases}
\]
In particular, $f(k,d)\ge d-k$ for $d\ge2k+1$.
Chu~\cite{chu} subsequently proved that, for $m\ge2$, a
$k$-connected graph with minimum degree at least
$\max\{k+1,2m-2\}$ has a $k$-removable $m$-matching, except when
the graph is $K_{2m-1}$ or when $(k,m)=(1,2)$ and the graph is a
cycle. His result yields
$f(k,d)\ge\lceil(d+1)/2\rceil$ whenever $d>k\ge1$ and
$(k,d)\ne(1,2)$.

For $k=2$, Li, Zhou, Fujita, and Mao~\cite{lzfm} proved that
$f(2,d)\ge\lceil(d+1)/2\rceil$ for $d\ge3$ and determined
$f(2,3)=2$. The bound $f(2,d)\ge d-2$ for $d\ge5$ follows
from~\cite{lzfm} for even $d$ and from~\cite{ckp} in general.

Our main result shows that the upper bound $f(2,d)\le d$
is attained for every $d\ge4$.
From now on, \emph{removable} means \emph{$2$-removable}.
\begin{theorem}\label{thm:main}
Let $d\ge4$ be an integer, and let $G$ be a $2$-connected graph
with $\delta(G)\ge d$. If $|V(G)|\ge2d$, then $G$ has a removable
$d$-matching.
\end{theorem}

Together with the upper bound $f(2,d)\le d$ and the known
equality $f(2,3)=2$~\cite{lzfm}, Theorem~\ref{thm:main} gives
the following complete determination of $f(2,d)$.

\begin{corollary}\label{cor:f2d}
We have $f(2,3)=2$ and $f(2,d)=d$ for every integer $d\ge4$.
\end{corollary}

We also characterize the obstruction at minimum degree $3$.
For $n\ge4$, the \emph{wheel} $W_n$ is obtained from a cycle on
$n-1$ vertices by adding a vertex adjacent to every vertex of the
cycle. The added vertex is called the \emph{center}, the cycle
is called the \emph{rim}, and the edges incident with the center
are called \emph{spokes}.

\begin{theorem}\label{thm:3matching}
Let $G$ be a $2$-connected graph with $\delta(G)\ge3$ and
$|V(G)|\ge6$. Then $G$ has a removable $3$-matching if and only if
$G$ is not a wheel.
\end{theorem}

Li, Zhou, Fujita, and Mao~\cite[Conjecture 6.3]{lzfm}
conjectured that every $k$-connected graph $G$ with
$\delta(G)\ge k+2$ has a $k$-removable matching of size
$\min\{\delta(G),\lfloor |V(G)|/2\rfloor\}$.
The following consequence of Theorems~\ref{thm:main}
and~\ref{thm:3matching} confirms this conjecture for $k=2$.

\begin{corollary}\label{cor:all-orders}
Every $2$-connected graph $G$ with $\delta(G)\ge4$ has a removable
matching of size
\[
\min\left\{\delta(G),\left\lfloor\frac{|V(G)|}{2}\right\rfloor\right\}.
\]
\end{corollary}
\begin{proof}
Let $n=|V(G)|$ and put
\[
r=\min\{\delta(G),\lfloor n/2\rfloor\}.
\]
If $r\ge4$, then $\delta(G)\ge r$ and $n\ge2r$, so
Theorem~\ref{thm:main} gives a removable $r$-matching.

Otherwise, $\delta(G)\ge4$ implies that $n\in\{5,6,7\}$.
If $n=5$, then $G=K_5$, which has a removable $2$-matching.
If $n\in\{6,7\}$, then $r=3$ and $G$ is not a wheel because
$\delta(G)\ge4$. Hence Theorem~\ref{thm:3matching} gives a
removable $3$-matching.
\end{proof}
The bound is best possible in general since every matching has
at most $\lfloor |V(G)|/2\rfloor$ edges, and $K_{d,n-d}$
has no matching of size $d+1$ when $n\ge2d$.

Our proofs use \emph{minimal non-removable matchings},
that is, matchings that are not removable but whose proper
submatchings are all removable.
We describe the block structure of $G-M$ for such a matching
$M$ and use edge contraction and induction to prove both
theorems.

The paper is organized as follows.
Section~\ref{sec:prelim} collects the notation and auxiliary lemmas.
Section~\ref{sec:minimal} establishes the structure and contraction
lemmas for minimal non-removable matchings.
Section~\ref{sec:proof} proves Theorem~\ref{thm:main},
starting with graphs on exactly $2d$ vertices.
Section~\ref{sec:degree-three} proves Theorem~\ref{thm:3matching}
and presents a conjecture for higher connectivity.

%=======================================================================
\section{Preliminaries and auxiliary lemmas}\label{sec:prelim}
%=======================================================================

For a vertex $v$ of a graph $G$, let $N_G(v)$ and $\deg_G(v)$
be its neighborhood and degree, respectively.
For $X\subseteq V(G)$, let $G-X$ be the graph obtained by
deleting the vertices in $X$ and all edges incident with them,
and write $G-v$ when $X=\{v\}$.
We write $G[X]$ for the subgraph induced by $X$.
For non-adjacent vertices $u,v\in V(G)$, let $G+uv$ be the
graph obtained by adding the edge $uv$.
For $uv\in E(G)$, let $G/uv$ be the simple graph obtained by
identifying $u$ and $v$, deleting the resulting loop, and
replacing parallel edges by a single edge.

We use the standard notation $K_n$ and $K_{a,b}$, and write
$K_n^-$ for $K_n$ with one edge deleted.
For disjoint sets $X,Y\subseteq V(G)$, an edge \emph{joins}
$X$ and $Y$ if it has one end in each set.
We say that $X$ is \emph{complete to} $Y$ if every vertex
of $X$ is adjacent to every vertex of $Y$.

For a matching $M$, let $V(M)$ be the set of its end vertices.
We say that $M$ \emph{avoids} a vertex set $W$ if
$V(M)\cap W=\varnothing$.
We allow the empty matching, which is removable in every
$2$-connected graph.
Every submatching of a removable matching is removable.

Let $H$ be a connected graph.
A \emph{cut-vertex} of $H$ is a vertex $v$ such that $H-v$
is disconnected.
For disjoint nonempty sets $X,Y\subseteq V(H)$, a vertex
$v\in V(H)\setminus(X\cup Y)$ \emph{separates} $X$ and $Y$
if every path in $H$ with one end in $X$ and the other in $Y$
contains $v$.
A \emph{separation with cut-vertex $v$} of $H$ is a pair
$(A,B)$ of vertex sets such that $A\cup B=V(H)$,
$A\cap B=\{v\}$, both $A\setminus B$ and $B\setminus A$
are nonempty, and no edge joins these two sets.

We use standard block terminology~\cite{diestelbook}.
A \emph{block} of a connected graph $H$ is a maximal connected
subgraph without a cut-vertex.
Each block of a connected graph on at least two vertices is
either $2$-connected or a copy of $K_2$.
The \emph{block tree} of $H$ is the bipartite incidence graph
whose vertices are the blocks and cut-vertices of $H$,
with a block adjacent to each cut-vertex it contains.
This graph is a tree.
A \emph{leaf block} is a block containing exactly one
cut-vertex of $H$.
The \emph{interior} of a block consists of its vertices
that are not cut-vertices of $H$.

Each of the following operations on a $2$-connected graph $G$
preserves $2$-connectivity. We use these facts without further
mention.
\begin{itemize}
\item Taking a spanning supergraph, that is, adding edges
without changing the vertex set.
\item Adding a new vertex adjacent to at least two vertices
of $G$.
\end{itemize}
In particular, every graph containing a Hamiltonian cycle
is $2$-connected.

\begin{lemma}\label{lem:min_matching}
Every graph $G$ on $n$ vertices has a matching of size
$\min\{\delta(G),\lfloor n/2\rfloor\}$.
\end{lemma}
\begin{proof}
Let $M$ be a maximum matching of $G$, and let
$U=V(G)\setminus V(M)$. If $|U|\le1$, then
$|M|=\lfloor n/2\rfloor$, so we may assume that $|U|\ge2$.
Since $M$ is maximum, the set $U$ is independent.

Suppose that $|M|<\delta(G)$, and choose distinct $u,u'\in U$.
All neighbors of $u$ and $u'$ lie in $V(M)$, and
\[
\deg_G(u)+\deg_G(u')\ge2\delta(G)>2|M|.
\]
Hence there is an edge $xy\in M$ such that at least three edges
join $\{x,y\}$ and $\{u,u'\}$. After interchanging $x$ and $y$
if necessary, we have $ux,yu'\in E(G)$.
Then $(M\setminus\{xy\})\cup\{ux,yu'\}$ is a larger matching,
a contradiction. Thus $|M|\ge\delta(G)$. This completes the proof.
\end{proof}

The following is the case $k=2$ of a theorem of Chu.

\begin{theorem}[Chu~\cite{chu}]\label{thm:prescribed-set}
    Let $H$ be a $2$-connected graph, and let $W\subseteq V(H)$.
    Suppose that every vertex outside $W$ has degree at least $\max\{3,|W|\}$ in $H$ and 
    that $E(H-W)\ne\varnothing$. 
    Then $H$ has a removable edge with both ends outside $W$, 
    unless $|W|=3$ and $H-W$ is a forest whose leaves all have degree $3$
in $H$.
\end{theorem}

We use the following observation without proof. See Figure~\ref{fig:leafblocks}.
\begin{observation}\label{obs:leaf-blocks}
Let $H$ be a connected graph whose block tree is a path
of length at least two.
Let $X_1$ and $X_2$ be its leaf blocks, and let $u$ and $v$
be interior vertices of $X_1$ and $X_2$, respectively. Then $H+uv$ is
$2$-connected.
\end{observation}

\begin{figure}[t] 
\centering
\begin{tikzpicture}[scale=1.05,
    block/.style={draw, fill=gray!12},
    vtx/.style={circle, fill=black, inner sep=1.4pt},
    newedge/.style={thick, dashed}]

    \draw[block] (-1.70,0.6) ellipse (0.95 and 0.62);
    \draw[block] (0,0.6)     ellipse (0.75 and 0.48);
    \draw[block] (1.70,0.6)  ellipse (0.95 and 0.62);
    \node[vtx] at (-0.75,0.6) {};
    \node[vtx] at (0.75,0.6)  {};
    \node[vtx, label={[label distance=-1pt]270:{\small $u$}}] (w1) at (-1.95,0.7) {};
    \node[vtx, label={[label distance=-1pt]270:{\small $v$}}]  (w2) at (1.95,0.7) {};
    \node at (-1.70,-0.3) {\small $X_1$};
    \node at (1.70,-0.3)  {\small $X_2$};
    \draw[newedge] (w1) to[bend left=50] (w2);
\end{tikzpicture}
\caption{Adding the dashed edge $uv$ between interior vertices
of the two leaf blocks makes $H$ $2$-connected when its block
tree is a path.}
\label{fig:leafblocks}
\end{figure}

\begin{lemma}\label{lem:contract}
Let $G$ be a graph on at least four vertices, and let
$uv\in E(G)$.  
If $G/uv$ is $2$-connected and $\deg_G(u),\deg_G(v)\ge2$, then $G$ is $2$-connected.
\end{lemma}
\begin{proof}
Let $Q=G/uv$.
Clearly, $G$ is connected.
Suppose $G$ has a cut-vertex $c$, 
and let $(A,B)$ be a separation of $G$ with cut-vertex $c$.
If $c\notin \{u,v\}$, then $c$ is a cut-vertex of $Q$, a contradiction. 
By symmetry, we may assume $c=u$ and $v\in A\setminus\{u\}$. 
In order for $Q$ to be $2$-connected, it must be that $A=\{u,v\}$, a contradiction since $\deg_G(v)\ge 2$.
Hence, $G$ is $2$-connected.
\end{proof}

%=======================================================================
\section{Minimal non-removable matchings}\label{sec:minimal}
%=======================================================================

A matching $M$ of a $2$-connected graph $G$ is \emph{minimal
non-removable} if it is not removable but $M\setminus\{e\}$
is removable for every $e\in M$. Such a matching is nonempty.

\begin{lemma}\label{lem:structure}
Let $G$ be a $2$-connected graph with $\delta(G)\ge3$, let $M$
be a minimal non-removable matching, and put $H=G-M$.
Then $H$ is connected, its block tree is a path with 
two $2$-connected leaf blocks, and every edge of $M$
joins their interiors.
\end{lemma}
\begin{proof}
Since $M$ is a matching, $\delta(H)\ge\delta(G)-1\ge2$.

Fix $e\in M$.
By the minimality of $M$, the graph
$H+e=G-(M\setminus\{e\})$ is $2$-connected.
If $H$ is not connected, then $H+e$ is not $2$-connected, a contradiction. Hence $H$ is connected.
Since $H$ is not $2$-connected, it has at least two leaf blocks.

If a leaf block $B$ with cut-vertex $c$ contains no end of $e$
in its interior, then $c$ still separates that interior from
$V(H)\setminus V(B)$ in $H+e$, a contradiction.
Thus the interior of every leaf block contains an end of $e$.
Since the interiors are disjoint, there are exactly two
leaf blocks, and $e$ joins their interiors.
This holds for every $e\in M$.
As the block tree has exactly two leaves, it is a path.

For each leaf block $B$, every interior vertex $v$ satisfies 
$\deg_B(v)=\deg_H(v)\ge 2$.
Thus $B$ has at least three vertices and is $2$-connected.
\end{proof}

For the remainder of this section, let $G,M,H$ be as in
Lemma~\ref{lem:structure}, and let $X_1,X_2$ be the leaf
blocks of $H$.
For $i=1,2$, let $x_i$ be the unique cut-vertex of $H$
in $X_i$.
Every interior vertex $v$ of $X_i$ has all its neighbors
in $H$ within $X_i$.
Thus, if $\delta(G)\ge d$, then
\begin{equation}\label{eq:leaf-degrees}
\deg_{X_i}(v)\ge
\begin{cases}
d & \text{if }v\notin V(M),\\
d-1 & \text{if }v\in V(M).
\end{cases}
\end{equation}

\begin{lemma}\label{lem:matching-contraction}
Let $uv\in M$, and put $Q=G/uv$.
Then $Q$ is $2$-connected.
Moreover, if $Q$ has a removable $\ell$-matching, then
$G$ has a removable $\ell$-matching.
\end{lemma}
\begin{proof}
Let $w$ be the contracted vertex of $Q$.
We first show that $Q$ is $2$-connected.
By symmetry, we may assume that
$u$ and $v$ lie in the interiors of $X_1$ and $X_2$, respectively.
Since $X_1-u$ and $X_2-v$ are connected and contain their
respective cut-vertices, $H-\{u,v\}$ is connected.
Hence $Q-w=G-\{u,v\}$ is connected, that is, $w$ is not a cut-vertex of $Q$.

For every $z\in V(Q)\setminus\{w\}$, the graph $Q-z=(G-z)/uv$
is connected because $G-z$ is connected.
Thus $Q$ has no cut-vertex and is $2$-connected.

Now let $N$ be a removable $\ell$-matching of $Q$.
If $w\notin V(N)$, set $F=N$.
Otherwise, replace the unique edge $wy\in N$ by one of
$uy,vy$ that belongs to $E(G)$, and call the resulting
matching $F$.
Then $|F|=\ell$, and $(G-F)/uv$ contains $Q-N$ as a
spanning subgraph, so it is $2$-connected.
Since both $u$ and $v$ have degree at least two in $G-F$, 
Lemma~\ref{lem:contract} implies that $G-F$ is
$2$-connected.
Thus $F$ is a removable $\ell$-matching of $G$.
\end{proof}

\begin{lemma}\label{lem:assembly}
For each $2$-connected block $B$ of $H$, let $S_B\subseteq E(B)$
be such that $B-S_B$ is $2$-connected, and let $T\subsetneq M$.
If $F=T\cup\bigcup_B S_B$ is a matching, where the union is over
all $2$-connected blocks of $H$, then $F$ is removable in $G$.
\end{lemma}
\begin{proof}
Deleting $S_B$ inside each block preserves the block tree of
$H$ and the interiors of its leaf blocks.
Blocks isomorphic to $K_2$ are unchanged.
Choose $e\in M\setminus T$.
By Observation~\ref{obs:leaf-blocks}, adding $e$ to the
resulting graph makes it $2$-connected.
The graph $G-F$ is a spanning supergraph of this graph,
so it is $2$-connected.
\end{proof}

\begin{lemma}\label{lem:matching+edge}
Write $M=\{u_1v_1,\ldots,u_mv_m\}$ with $u_j$ in the interior
of $X_1$ and $v_j$ in the interior of $X_2$ for every $j$.
For $i=1,2$, let $N_i$ be a removable matching of $X_i$
avoiding $\{x_i\}\cup(V(M)\cap V(X_i))$.
For $1\le j\le m$, put
\[
M_j=(M\setminus\{u_jv_j\})\cup N_1\cup N_2.
\]
Then $M_j$ is a removable matching of $G$.
Moreover, if $x_1u_j\in E(X_1)$, then
$M_j\cup\{x_1u_j\}$ is also a removable matching of $G$.
The analogous statement holds for $x_2v_j$.
\end{lemma}
\begin{proof}
The first assertion follows from Lemma~\ref{lem:assembly}.
By symmetry, it remains to consider $f=x_1u_j\in E(X_1)$.
The set $M_j\cup\{f\}$ is a matching.
Put
\[
D=H-(N_1\cup N_2),
\qquad
J=G-(M_j\cup\{f\}).
\]
The graph $D$ has the same block tree as $H$.
Since $J+f=G-M_j$ is $2$-connected, $J$ is connected.

Suppose that $J$ has a cut-vertex $c$.
Then $c\notin\{x_1,u_j\}$, since deleting either end of $f$
from $J$ gives the same graph as deleting it from $J+f$.
Moreover, $x_1$ and $u_j$ lie in different components of $J-c$,
since adding $f$ makes $J-c$ connected.

There is an $x_1v_j$-path in $D$ meeting $X_1$ only at $x_1$.
Together with $v_ju_j$, it gives an $x_1u_j$-path in $J$
whose internal vertices are outside $X_1$.
Since $X_1-N_1$ is $2$-connected, it also has an
$x_1u_j$-path avoiding $f$, entirely in $X_1$.
These two paths are internally disjoint, so at least one
avoids $c$, a contradiction.
Thus $J$ is $2$-connected.
\end{proof}

%=======================================================================
\section{Proof of Theorem~\ref{thm:main}}\label{sec:proof}
%=======================================================================

In this section, we prove Theorem~\ref{thm:main}.
We first establish the case $|V(G)|=2d$ and then complete
the proof using a counterexample of minimum order.

\begin{lemma}\label{lem:order-2d}
Let $d\ge4$ be an integer, and let $G$ be a $2$-connected graph with
$\delta(G)\ge d$ and $|V(G)|=2d$. Then $G$ has a removable
perfect matching.
\end{lemma}
\begin{proof}
By Lemma~\ref{lem:min_matching}, $G$ has a perfect matching.
If it is removable, then it is the desired matching.
So, we may assume that it is not removable.
Then there exists a minimal non-removable matching $M$.
Let $H=G-M$.

By Lemma~\ref{lem:structure}, $H$ has two leaf blocks, say
$X_1$ and $X_2$, and every edge of $M$ joins their interiors.
Since $X_1$ has a vertex of $V(M)$ and every vertex of $V(M)$
has degree at least $d-1$ in $H$, we have $|V(X_1)|\ge d$.
Similarly, $|V(X_2)|\ge d$.
Since $X_1$ and $X_2$ have at most one common vertex and
$|V(H)|=2d$, at least one of them has exactly $d$ vertices.
By symmetry, we may assume that $|V(X_1)|=d$.
Then $X_1$ is a copy of $K_d$, and every interior vertex of
$X_1$ is matched to a vertex of $X_2$ by $M$.
Thus $|M|=d-1$.

Suppose that $X_1$ and $X_2$ are disjoint.
Since $|V(H)|=2d$, $X_2$ is a copy of $K_d$.
Furthermore, $H$ consists of $X_1$, $X_2$, and an edge $x_1x_2$,
where $x_i\in V(X_i)$ for $i=1,2$.
Let $uv\in M$, with $u\in V(X_1)$ and $v\in V(X_2)$, and put
\[
F=(M\setminus\{uv\})\cup\{x_1u,x_2v\}.
\]
Since $X_1-x_1u$ and $X_2-x_2v$ are copies of $K_d^-$,
there are Hamiltonian paths $L_1$ in $X_1-x_1u$ from $x_1$ to $u$
and $L_2$ in $X_2-x_2v$ from $x_2$ to $v$.
Hence $F$ is a removable perfect matching of $G$, since
$G-F$ contains a Hamiltonian cycle consisting of $L_1$, $L_2$,
and the edges $x_1x_2$ and $uv$.

Suppose now that $X_1$ and $X_2$ are not disjoint.
Then they are the only blocks of $H$ and share a cut-vertex $x$.
Since $|V(X_1)|=d$, we have $|V(X_2)|=d+1$.
Every vertex of $V(X_1)\setminus\{x\}$ is covered by $M$,
and there is a unique vertex $t\in V(X_2)\setminus\{x\}$
not covered by $M$.
As $\deg_{X_2}(t)=\deg_G(t)\ge d$, the vertex $t$ is adjacent
to every other vertex of $X_2$.

Let $X_2'=X_2-t$.
Then $\deg_{X_2'}(x)\ge1$ and $\deg_{X_2'}(y)\ge d-2$
for all $y\in V(X_2'-x)$.
Let $L=p_1p_2\cdots p_\ell$ be a longest path in $X_2'$
starting at $p_1=x$.
All neighbors of $p_\ell$ lie on $L$, so $\ell\ge d-1$.

Suppose that $\ell=d-1$, and let $z$ be the unique vertex
of $X_2'$ outside $L$.
Then $p_\ell$ is adjacent to every other vertex of $L$.
By the choice of $L$, the vertices $z$ and $p_\ell$
are not adjacent.
Since $\deg_{X_2'}(z)\ge d-2$, the vertex $z$ is adjacent
to every vertex of $L-p_\ell$.
Thus $p_1p_\ell p_{\ell-1}\cdots p_2z$ is a longer path
starting at $x$, a contradiction.
Therefore $L$ is Hamiltonian in $X_2'$.

Let $v=p_d$.
Inserting $t$ immediately after $x$ in $L$ gives a
Hamiltonian path $L_2$ from $x$ to $v$ in $X_2-tv$.
Let $u\in V(X_1)\setminus\{x\}$ satisfy $uv\in M$, and put
\[
F=(M\setminus\{uv\})\cup\{xu,tv\}.
\]
Then $F$ is a perfect matching of $G$.
Since $X_1\cong K_d$, there is a Hamiltonian path $L_1$ 
from $x$ to $u$ in $X_1-xu$.
The paths $L_1$ and $L_2$, together with the edge $uv$,
form a Hamiltonian cycle in $G-F$.
Hence $F$ is removable.
\end{proof}

We next establish two auxiliary lemmas for the proof of
Theorem~\ref{thm:main}.

\begin{lemma}\label{lem:bipartite}
Let $G$ be a $2$-connected graph, and let $\{X,Y\}$ be a partition
of $V(G)$ such that $X$ is complete to $Y$ and $|Y|\ge3$.
If $x\in X$, $y\in Y$, and $\deg_G(y)\ge3$, then $xy$ is removable.
\end{lemma}
\begin{proof}
Put $J=G-xy$.
Suppose that $J-c$ is disconnected for some $c\in V(J)$.
Since $J-x=G-x$ and $J-y=G-y$ are connected,
$c\notin\{x,y\}$.
Since $G-c=(J-c)+xy$ is connected, $J-c$ has exactly two
components $J_1,J_2$ containing $x,y$, respectively.

Since $\deg_{J-c}(y)\ge\deg_G(y)-2\ge1$, the vertex $y$
has a neighbor $x'\in V(J_2)$.
If $x'\in Y$, then $xx'\in E(J-c)$, a contradiction.
Thus $x'\in X$.
If $J_1$ contains a vertex of $Y$, then this vertex is
adjacent to $x'$ in $J-c$, a contradiction.
Hence $V(J_1)\cap Y=\varnothing$.

Since $x$ is adjacent in $J-c$ to every vertex of
$Y\setminus\{c,y\}$, we have $V(J_2)\cap Y=\{y\}$.
Consequently,
\[
|Y|\le |V(J_1)\cap Y|+|V(J_2)\cap Y|+1=2,
\]
contrary to $|Y|\ge3$.
Therefore $J-c$ is connected for every $c\in V(J)$,
so $J$ is $2$-connected.
\end{proof}

\begin{lemma}\label{lem:block}
Let $d\ge4$ be an integer, let $G$ be a $2$-connected graph,
and let $W\subsetneq V(G)$ such that every vertex outside
$W$ has degree at least $d$ in $G$.
Let $N$ be a maximum removable matching avoiding $W$, and put
$P=V(G)\setminus(W\cup V(N))$.
Then $|W\cup V(N)|\ge d$.
If equality holds, then $P$ is independent.
\end{lemma}

\begin{proof}
Put $H=G-N$ and $S=W\cup V(N)$.
Then $H$ is $2$-connected, and
\[
\deg_H(v)=\deg_G(v)\ge d
\qquad\text{for every }v\in P.
\]
We may assume that $|S|\le d$, since otherwise there is
nothing to prove.

Suppose $P$ is not independent.
Since $\deg_H(v)\ge d\ge\max\{3,|S|\}$ for every $v\in P$,
we may apply Theorem~\ref{thm:prescribed-set} to $H$ and $S$.
If $H[P]$ is a forest, then it has a leaf because $P$ is not independent, 
and such a leaf has degree at least $d\ge 4$ in $H$.
Thus the exceptional case cannot occur, and $H$ has a
removable edge $e$ with both ends in $P$.
Then $N\cup\{e\}$ is a removable matching in $G$ avoiding $W$,
contrary to the choice of $N$.
Therefore $P$ is independent.

Since $W\subsetneq V(G)$ and every vertex outside $W$
has degree at least $d$, we have $|V(G)|\ge d+1>|S|$.
Hence $P\ne\varnothing$.
For any $v\in P$, the independence of $P$ gives
\[
d\le\deg_G(v)\le |S|\le d.
\]
Thus $|S|=d$, completing the proof.
\end{proof}

Now we are ready to prove Theorem~\ref{thm:main}.

\begin{proof}[Proof of Theorem~\ref{thm:main}]
Fix an integer $d\ge4$ and choose a counterexample $G$ of
minimum order.
By Lemma~\ref{lem:order-2d}, $|V(G)|\ge2d+1$.
By Lemma~\ref{lem:min_matching}, $G$ has a $d$-matching.
This matching is not removable, so it contains a minimal
non-removable matching $M$ of size $m\le d$.
Put $H=G-M$.
By Lemma~\ref{lem:structure}, the block tree of $H$ is a path
with two $2$-connected leaf blocks $X_1,X_2$, and every edge
of $M$ joins their interiors.

We first show that $H$ has exactly two blocks sharing a
cut-vertex $x$ of degree $d$ in $G$, and that $x$ is adjacent
to every end vertex of $M$.

Fix an arbitrary edge $x_1x_2\in M$, with $x_i$ in the
interior of $X_i$ for $i=1,2$.
Put $Q=G/x_1x_2$, and let $w$ be the contracted vertex.
By Lemma~\ref{lem:matching-contraction}, $Q$ is $2$-connected.
If $\delta(Q)\ge d$, then the minimality of $G$ gives a
removable $d$-matching in $Q$, since $|V(Q)|\ge2d$.
By Lemma~\ref{lem:matching-contraction}, this yields a
removable $d$-matching in $G$, a contradiction.
Thus $\delta(Q)<d$.

Every edge incident with $x_1$ or $x_2$, other than $x_1x_2$,
belongs to $H$.
Since $x_1$ and $x_2$ are interior vertices of $X_1$ and $X_2$,
respectively, their common neighbors belong to
$V(X_1)\cap V(X_2)$.
In particular, they have at most one common neighbor, so
\[
\deg_Q(w)
=\deg_G(x_1)+\deg_G(x_2)-2
-|N_G(x_1)\cap N_G(x_2)|
\ge2d-3\ge d.
\]
Every other vertex retains its degree unless it is a common
neighbor of $x_1$ and $x_2$, in which case its degree
decreases by one.
Since $\delta(Q)<d$, the blocks $X_1$ and $X_2$ must
therefore share a cut-vertex $x$ satisfying
$\deg_G(x)=d$ and $xx_1,xx_2\in E(G)$.
Since $x_1x_2$ was chosen arbitrarily and $x$ is the unique
common vertex of $X_1$ and $X_2$, the vertex $x$ is adjacent
to every end vertex of $M$.
Since $X_1$ and $X_2$ share $x$ and the block tree of $H$
is a path, they are the only blocks of $H$.
In particular, $2m\le d$ since $|V(M)|\le\deg_G(x)=d$.

For $i=1,2$, let $W_i=\{x\}\cup(V(X_i)\cap V(M))$.
Note that $|W_i|=m+1\le d/2+1<d$ since $2m\le d$ and $d\ge4$.
Since $\deg_{X_i}(x_i)\ge d-1$, we have
$|V(X_i)|\ge d>|W_i|$.
Thus $W_i\subsetneq V(X_i)$.

For $i=1,2$, choose a maximum removable matching $N_i$
of $X_i$ avoiding $W_i$, and put
\[
s_i=|W_i\cup V(N_i)|=2|N_i|+m+1,\qquad
P_i=V(X_i)\setminus(W_i\cup V(N_i)).
\]
Every vertex of $V(X_i)\setminus W_i$ is an interior vertex
not covered by $M$, so its degree in $X_i$ equals its degree
in $G$ and is at least $d$.
In particular, $|V(X_i)|\ge d+1$ and it follows by
Lemma~\ref{lem:block} that $s_i\ge d$.

Put
\[
F=(M\setminus\{x_1x_2\})\cup N_1\cup N_2.
\]
We claim that $s_i=d$ for $i=1,2$.
Suppose this is false.
Then $s_1+s_2\ge2d+2$, since $s_1+s_2$ is even.
By Lemma~\ref{lem:matching+edge},
$F\cup\{xx_1\}$ is a removable matching of $G$.
However, it has size
\[
m+|N_1|+|N_2|=\frac{s_1+s_2-2}{2}\ge d,
\]
a contradiction.
Thus $s_1=s_2=d$.

Since $|V(X_i)|\ge d+1$ and $s_i=d$, each $P_i$ is nonempty.
The equality case of Lemma~\ref{lem:block} implies that
$P_i$ is independent.
Every vertex of $P_i$ has degree at least $d$ in $X_i$,
so $P_i$ is complete to the set $W_i\cup V(N_i)$ of size $d$.
This remains true in $X_i-N_i$.

Choose $p_i\in P_i$ for $i=1,2$.
Since $N_i$ avoids $x_i$, we have
\[
\deg_{X_i-N_i}(x_i)\ge d-1\ge3.
\]
Apply Lemma~\ref{lem:bipartite} to each $X_i-N_i$ with
the partition $\{P_i,W_i\cup V(N_i)\}$.
It follows that $p_1x_1$ is removable in $X_1-N_1$
and $p_2x_2$ is removable in $X_2-N_2$.
These two edges are vertex-disjoint and have no end in
$V(F)\cup\{x\}$.
Lemma~\ref{lem:assembly} now shows that
\[
F\cup\{p_1x_1,p_2x_2\}
\]
is a removable matching of size $d$, a contradiction.
\end{proof}

%=======================================================================
\section{Removable \texorpdfstring{$3$}{3}-matchings and further questions}
\label{sec:degree-three}
%=======================================================================

\subsection{Proof of Theorem~\ref{thm:3matching}}

We first establish two properties of wheels.

\begin{lemma}\label{lem:wheel-properties}
The following hold.
\begin{enumerate}[(i)]
\item For every $n\ge4$, the wheel $W_n$ has no
      removable $3$-matching.
\item If $n\ge6$ and a graph $G$ properly contains
      $W_n$ as a spanning subgraph, then $G$ has a removable
      $3$-matching.
\end{enumerate}
\end{lemma}
\begin{proof}
Let $c$ be the center of $W_n$ and
$r_1r_2\cdots r_{n-1}r_1$ its rim.
\begin{enumerate}[(i)]
\item A matching in $W_n$ contains at most one spoke, so every
      $3$-matching $M$ contains at least two rim edges.
      Then $c$ is a cut-vertex of $W_n-M$, and $M$ is not
      removable.

\item Since the center is adjacent to every other vertex,
      some edge of $G$ outside $W_n$ joins two rim vertices.
      By relabeling, we may assume that it is $r_1r_i$ with
      $3\le i\le\lfloor(n+1)/2\rfloor$.
      Then
      \[
      F=\{r_1r_i,cr_2,r_{n-2}r_{n-1}\}
      \]
      is a removable $3$-matching, since $G-F$ contains the
      Hamiltonian cycle
      $cr_{n-1}r_1r_2\cdots r_{n-2}c$. \qedhere
\end{enumerate}
\end{proof}

\begin{lemma}\label{lem:wheel}
Let $G$ be a graph with $\delta(G)\ge3$, and let $uv\in E(G)$.
If $G/uv$ is a wheel on at least five vertices, then $G$ is
a wheel or has a removable $3$-matching.
\end{lemma}
\begin{proof}
Let $w$ be the contracted vertex of $G/uv$, let $c$ be its
center, and write its rim as $r_1r_2\cdots r_nr_1$, where
$n\ge4$.
We may assume that $G$ is not a wheel.

Suppose first that $w=c$, and let
$A=N_G(u)\setminus\{v\}$ and $B=N_G(v)\setminus\{u\}$.
Both are subsets of the rim with at least two elements,
and together they cover the rim.
We call a rim edge \emph{mixed} if one end lies in $A$
and the other in $B$.
Since $A$ and $B$ cover the rim, there is a mixed edge.

We claim that there are two mixed edges sharing no end vertex.
Suppose not.
Since the rim is a cycle of length at least four, all mixed
edges are incident with a common rim vertex, say $r_1$.
By reflecting the rim and exchanging $u$ and $v$, we may
assume that $r_1r_2$ is mixed with $r_1\in A$ and $r_2\in B$.
If $r_2\in A$, then $r_2\in A\cap B$, so $r_2r_3$ is mixed,
a contradiction.
Thus $r_2\notin A$.
Let $j>1$ be the smallest index with $r_j\in A$, which
exists since $|A|\ge2$.
Then $j\ge3$, and $r_{j-1}\in B$ by the minimality of $j$.
Hence $r_{j-1}r_j$ is a mixed edge not incident with $r_1$,
a contradiction.
This proves the claim.

By relabeling, we may assume that $r_1r_2$ and $r_ir_{i+1}$
are mixed edges, where $3\le i\le n-1$.
By exchanging $u$ and $v$ if necessary, we may also assume
that $r_1\in A$ and $r_2\in B$.
Put
\[
F=\{ur_1,vr_2,r_ir_{i+1}\}.
\]
If $r_i\in A$ and $r_{i+1}\in B$, then $G-F$ contains the
Hamiltonian cycle
\[
vr_{i+1}\cdots r_nr_1\cdots r_iuv.
\]
Otherwise, $r_i\in B$ and $r_{i+1}\in A$, and $G-F$
contains the Hamiltonian cycle
\[
ur_{i+1}\cdots r_nr_1\cdots r_ivu.
\]
In either case, $F$ is a removable $3$-matching.

Now suppose that $w\ne c$.
By symmetry, we may assume that $w=r_n$.
Each of $u,v$ is adjacent to at least two of
$r_1,r_{n-1},c$, and each of $r_1,r_{n-1},c$ is adjacent
to at least one of $u,v$.
After exchanging $u,v$ and reflecting the rim if necessary,
we may assume that $ur_1,vr_{n-1},vc\in E(G)$.
If $uc\in E(G)$, then $G$ properly contains the spanning
wheel with center $c$ and rim
$r_1r_2\cdots r_{n-1}vur_1$, so
Lemma~\ref{lem:wheel-properties}(ii) applies.
Otherwise, $\deg_G(u)\ge3$ forces $ur_{n-1}\in E(G)$.
Then
\[
F=\{ur_{n-1},r_1r_2,cv\}
\]
is a removable $3$-matching, since $G-F$ contains the
Hamiltonian cycle $cr_1uvr_{n-1}r_{n-2}\cdots r_2c$.
\end{proof}

Now we are ready to prove Theorem~\ref{thm:3matching}.

\begin{proof}[Proof of Theorem~\ref{thm:3matching}]
Lemma~\ref{lem:wheel-properties}(i) shows that no wheel has
a removable $3$-matching.
To prove the converse, suppose that it fails, and choose
a counterexample $G$ of minimum order.
Thus $G$ is not a wheel and has no removable $3$-matching.

By Lemma~\ref{lem:min_matching}, $G$ has a $3$-matching.
This matching is not removable, so it contains a minimal
non-removable matching $M$.
Let $uv\in M$, put $Q=G/uv$, and let $w$ be the contracted
vertex.
Put $H=G-M$.
By Lemma~\ref{lem:matching-contraction}, $Q$ is $2$-connected.

We claim that $\delta(Q)\ge3$.
We only need to check the degrees of $w$ and the common 
neighbors of $u$ and $v$.
By Lemma~\ref{lem:structure}, $H$ has two $2$-connected
leaf blocks $X_1,X_2$, with $u$ and $v$ in their respective
interiors.
Thus  $N_G(u)\cap N_G(v)\subseteq V(X_1)\cap V(X_2)$ and  
$u$ and $v$ have at most one common neighbor.
So, $\deg_Q(w)=\deg_G(u)+\deg_G(v)-2-|N_G(u)\cap N_G(v)|\ge3$.
If $u$ and $v$ have a common neighbor $x$, then $x$ is the common cut-vertex of $X_1$ and $X_2$.
Thus $\deg_G(x)\ge\deg_{X_1}(x)+\deg_{X_2}(x)\ge4$, and $\deg_Q(x)\ge 3$. 
Therefore $\delta(Q)\ge3$.

If $|V(G)|\ge7$, then the minimality of $G$ implies that
$Q$ is a wheel or has a removable $3$-matching.
In the first case, Lemma~\ref{lem:wheel} gives a removable
$3$-matching in $G$, since $G$ is not a wheel.
In the second case, Lemma~\ref{lem:matching-contraction}
gives a removable $3$-matching in $G$.
Both cases yield a contradiction.
Thus $|V(G)|=6$.

Again by Lemma~\ref{lem:wheel}, $Q$ is not a wheel.
Since $Q$ has five vertices and minimum degree at least $3$,
it follows that $Q\cong K_5$ or $Q\cong K_5^-$.

If $Q\cong K_5$, choose two disjoint edges $e,f$ of $Q-w$,
and put $G'=G-e-f$.
If $Q\cong K_5^-$, choose an edge $e$ of $Q-w$ whose ends both have degree $4$ in $Q$, and put $G'=G-e$.
In either case, the ends of the deleted edges have degree
at least $4$ in $G$, so $\delta(G')\ge3$.
Moreover, $G'/uv\cong W_5$.

By Lemma~\ref{lem:wheel}, $G'$ is a wheel or has a
removable $3$-matching.
If $G'$ is a wheel, then $G'\cong W_6$, and $G$ properly
contains it as a spanning subgraph.
Thus Lemma~\ref{lem:wheel-properties}(ii) gives a removable
$3$-matching in $G$.
Otherwise, a removable $3$-matching of $G'$ is also
removable in $G$, since $G$ is a spanning supergraph of $G'$.
In either case, we obtain a contradiction.
\end{proof}

%=======================================================================
\subsection{A conjecture for higher connectivity}
\label{sec:conclusion}
%=======================================================================

For graphs $G_1$ and $G_2$, the \emph{join} of $G_1$ and $G_2$
is the graph obtained from their disjoint union by adding
all edges between $V(G_1)$ and $V(G_2)$.
A wheel is the join of a cycle and a single vertex.
The following observation generalizes
Lemma~\ref{lem:wheel-properties}(i).

\begin{observation}\label{obs:join}
Let $k\ge2$ be an integer, and let $G$ be the join of a cycle
$C$ and a graph $T$ on $k-1$ vertices, with $|V(G)|\ge2k+2$.
Then $G$ is $k$-connected with $\delta(G)=k+1$, and $G$ has
no $k$-removable $(k+1)$-matching.
\end{observation}
\begin{proof}
Every vertex of $C$ has degree $k+1$ in $G$, and every
vertex of $T$ has degree at least
$|V(C)|=|V(G)|-k+1\ge k+3$.
Thus $\delta(G)=k+1$.

Let $S\subseteq V(G)$ with $|S|<k$.
If $S=V(T)$, then $G-S$ is a cycle.
Otherwise, $T-S$ and $C-S$ are both nonempty, and every
vertex of $T-S$ is adjacent to every vertex of $C-S$.
Thus $G-S$ is connected, and $G$ is $k$-connected.

Let $M$ be a $(k+1)$-matching of $G$.
At most $k-1$ edges of $M$ are incident with $V(T)$, so
$M$ contains at least two edges of $C$.
Thus
\[
(G-M)-V(T)=C-(M\cap E(C))
\]
is disconnected, and $M$ is not $k$-removable.
\end{proof}

Li, Zhou, Fujita, and Mao~\cite{lzfm} used the case where
$T$ is an independent set to prove that $f(k,k+1)\le k$.
Their Conjecture~6.3 implies that $f(k,d)=d$ for every
$d\ge k+2$.
We conjecture that the graphs in Observation~\ref{obs:join}
are the only exceptions in the remaining case $d=k+1$.

\begin{conjecture}\label{con:join}
Let $k\ge2$ be an integer, and let $G$ be a $k$-connected
graph with $\delta(G)\ge k+1$ and $|V(G)|\ge2k+2$.
Then $G$ has a $k$-removable $(k+1)$-matching if and only if
$G$ is not the join of a cycle and a graph on $k-1$ vertices.
\end{conjecture}

Theorem~\ref{thm:3matching} is the case $k=2$ of
Conjecture~\ref{con:join}.

\section*{Acknowledgements}
Ringi Kim was supported by the National Research Foundation of Korea
(NRF) grant funded by the Korea government (MSIT)
(No.~RS-2025-00561867), and by an INHA UNIVERSITY Research Grant.

\section*{Declaration of generative AI use}
During the preparation of this work, the authors used OpenAI's
ChatGPT to check and refine parts of the proofs and improve
the exposition.
The authors take full responsibility for the content of the article.

\bibliographystyle{abbrv}
\bibliography{ref}

@book{diestelbook,
  author    = {Diestel, Reinhard},
  title     = {Graph Theory},
  edition   = {Sixth},
  series    = {Graduate Texts in Mathematics},
  volume    = {173},
  publisher = {Springer, Berlin},
  year      = {2025}
}

@article{lzfm,
  title   = {From {H}alin's edge removability to matching removability in $k$-connected graphs},
  author  = {Li, Hengzhe and Zhou, Mingming and Fujita, Shinya and Mao, Yaping},
  journal = {arXiv preprint arXiv:2605.24035},
  year    = {2026}
}

@article{cj,
  title   = {Removable trees and matchings in $k$-connected and $k$-edge-connected graphs},
  author  = {Clay, Adam D. W. and Jord{\'a}n, Tibor},
  journal = {arXiv preprint arXiv:2608.03643},
  year    = {2026}
}

@article {halin,
    AUTHOR = {Halin, R.},
     TITLE = {A theorem on {$n$}-connected graphs},
   JOURNAL = {J. Combinatorial Theory},
  FJOURNAL = {Journal of Combinatorial Theory},
    VOLUME = {7},
      YEAR = {1969},
     PAGES = {150--154},
      ISSN = {0021-9800},
   MRCLASS = {05.40},
  MRNUMBER = {248042},
MRREVIEWER = {I.\ Z.\ Bouwer},
}

@article {hasunuma,
    AUTHOR = {Hasunuma, Toru},
     TITLE = {Connectivity preserving trees in {$k$}-connected or
              {$k$}-edge-connected graphs},
   JOURNAL = {J. Graph Theory},
  FJOURNAL = {Journal of Graph Theory},
    VOLUME = {102},
      YEAR = {2023},
    NUMBER = {3},
     PAGES = {423--435},
      ISSN = {0364-9024,1097-0118},
   MRCLASS = {05C40 (05C05)},
  MRNUMBER = {4563199},
MRREVIEWER = {Stephen\ G.\ Hartke},
       DOI = {10.1002/jgt.22878},
       URL = {https://doi.org/10.1002/jgt.22878},
}

@article{chu,
  title   = {A sharp extension of {H}alin's removable-edge theorem to matchings},
  author  = {Chu, Hojin},
  journal = {arXiv preprint arXiv:2608.09394},
  year    = {2026}
}

@article{ckp,
  title={Minimum degree conditions for removable matchings in $k$-connected graphs},
  author={Chu, Hojin and Kim, Ringi and Park, Boram},
  journal={arXiv preprint arXiv:2607.17533},
  year={2026}
}

@article {ckl,
    AUTHOR = {Chartrand, Gary and Kaugars, Agnis and Lick, Don R.},
     TITLE = {Critically {$n$}-connected graphs},
   JOURNAL = {Proc. Amer. Math. Soc.},
  FJOURNAL = {Proceedings of the American Mathematical Society},
    VOLUME = {32},
      YEAR = {1972},
     PAGES = {63--68},
      ISSN = {0002-9939,1088-6826},
   MRCLASS = {05C99},
  MRNUMBER = {290999},
MRREVIEWER = {L.\ H.\ Harper},
       DOI = {10.2307/2038307},
       URL = {https://doi.org/10.2307/2038307},
}

@article {fk,
    AUTHOR = {Fujita, Shinya and Kawarabayashi, Ken-ichi},
     TITLE = {Connectivity keeping edges in graphs with large minimum
              degree},
   JOURNAL = {J. Combin. Theory Ser. B},
  FJOURNAL = {Journal of Combinatorial Theory. Series B},
    VOLUME = {98},
      YEAR = {2008},
    NUMBER = {4},
     PAGES = {805--811},
      ISSN = {0095-8956,1096-0902},
   MRCLASS = {05C40},
  MRNUMBER = {2418773},
MRREVIEWER = {Akira\ Saito},
       DOI = {10.1016/j.jctb.2007.11.001},
       URL = {https://doi.org/10.1016/j.jctb.2007.11.001},
}

@article {hl,
    AUTHOR = {Hong, Yanmei and Liu, Qinghai},
     TITLE = {{M}ader's conjecture for graphs with small connectivity},
   JOURNAL = {J. Graph Theory},
  FJOURNAL = {Journal of Graph Theory},
    VOLUME = {101},
      YEAR = {2022},
    NUMBER = {3},
     PAGES = {379--388},
      ISSN = {0364-9024,1097-0118},
   MRCLASS = {05C75 (05C05 05C35)},
  MRNUMBER = {4512138},
MRREVIEWER = {Audace\ Amen V. Dossou-Olory},
       DOI = {10.1002/jgt.22831},
       URL = {https://doi.org/10.1002/jgt.22831},
}

@article {llh,
    AUTHOR = {Liu, Haiyang and Liu, Qinghai and Hong, Yanmei},
     TITLE = {Connectivity keeping trees in 3-connected or 3-edge-connected
              graphs},
   JOURNAL = {Discrete Math.},
  FJOURNAL = {Discrete Mathematics},
    VOLUME = {346},
      YEAR = {2023},
    NUMBER = {12},
     PAGES = {Paper No. 113679, 4},
      ISSN = {0012-365X,1872-681X},
   MRCLASS = {05C40},
  MRNUMBER = {4634204},
       DOI = {10.1016/j.disc.2023.113679},
       URL = {https://doi.org/10.1016/j.disc.2023.113679},
}

@article {mader-path,
    AUTHOR = {Mader, W.},
     TITLE = {Connectivity keeping paths in {$k$}-connected graphs},
   JOURNAL = {J. Graph Theory},
  FJOURNAL = {Journal of Graph Theory},
    VOLUME = {65},
      YEAR = {2010},
    NUMBER = {1},
     PAGES = {61--69},
      ISSN = {0364-9024,1097-0118},
   MRCLASS = {05C40},
  MRNUMBER = {2682514},
MRREVIEWER = {Shinya\ Fujita},
       DOI = {10.1002/jgt.20465},
       URL = {https://doi.org/10.1002/jgt.20465},
}

@article {mader-tree,
    AUTHOR = {Mader, W.},
     TITLE = {Connectivity keeping trees in {$k$}-connected graphs},
   JOURNAL = {J. Graph Theory},
  FJOURNAL = {Journal of Graph Theory},
    VOLUME = {69},
      YEAR = {2012},
    NUMBER = {3},
     PAGES = {324--329},
      ISSN = {0364-9024,1097-0118},
   MRCLASS = {05C40 (05C05)},
  MRNUMBER = {2898871},
MRREVIEWER = {Peter\ Dankelmann},
       DOI = {10.1002/jgt.20585},
       URL = {https://doi.org/10.1002/jgt.20585},
}

@article {tm,
    AUTHOR = {Tian, Yingzhi and Meng, Jixiang},
     TITLE = {A survey on the vertex-(edge-){$k$}-maximal graphs and the
              {$k$}-vertex-(edge-)connected graphs with redundant subgraphs},
   JOURNAL = {Discrete Appl. Math.},
  FJOURNAL = {Discrete Applied Mathematics. The Journal of Combinatorial
              Algorithms, Informatics and Computational Sciences},
    VOLUME = {378},
      YEAR = {2026},
     PAGES = {125--135},
      ISSN = {0166-218X,1872-6771},
   MRCLASS = {05C40},
  MRNUMBER = {4933151},
       DOI = {10.1016/j.dam.2025.07.008},
       URL = {https://doi.org/10.1016/j.dam.2025.07.008},
}

@article {yt,
    AUTHOR = {Yang, Qing and Tian, Yingzhi},
     TITLE = {Connectivity keeping edges of trees in 3-connected or
              3-edge-connected graphs},
   JOURNAL = {Discrete Math.},
  FJOURNAL = {Discrete Mathematics},
    VOLUME = {347},
      YEAR = {2024},
    NUMBER = {5},
     PAGES = {Paper No. 113913, 5},
      ISSN = {0012-365X,1872-681X},
   MRCLASS = {05C40 (05C05)},
  MRNUMBER = {4702527},
MRREVIEWER = {Yan\ Zhao},
       DOI = {10.1016/j.disc.2024.113913},
       URL = {https://doi.org/10.1016/j.disc.2024.113913},
}

\end{document}